\documentclass{amsart}
\usepackage{pdfsync}
\usepackage [latin1]{inputenc}
\usepackage[T1]{fontenc}
\usepackage{amsmath,amsthm,amssymb}
\usepackage{enumerate}
\usepackage{epsfig,graphicx,graphics}
\usepackage{float}

\newcommand{\Z}{\mathbb Z}

\newtheorem{definition}{{\bf Definition}}[section]
\newtheorem{theorem}[definition]{{\bf Theorem}}

\newtheorem{proposition}[definition]{\noindent {\bf Proposition}}
\newtheorem{lemma}[definition]{\noindent {\bf Lemma}}

\newtheorem{claim}[definition]{\noindent {\bf Claim}}

\newtheorem{problem}[definition]{\noindent {\bf Problem}}

\def\endproof{\hfill {\kern 6pt\penalty 500
    \raise -0pt\hbox{\vrule \vbox to5pt {\hrule width 5pt
        \vfill\hrule}\vrule}}}

\usepackage{enumerate}

\begin{document}

\title[Claw-free graphs and  a reconstruction problem]
{Claw-freeness, $3$-homogeneous subsets of a graph and  a reconstruction problem}

\date{November 22, 2010}
\thanks{}

\author{Maurice Pouzet}
\address{ICJ, Universit\'e de Lyon, Universit\'e Claude Bernard Lyon 1, 43 boulevard du 11 novembre 1918, 69622 Villeurbanne cedex, France}
\curraddr{}
\email{pouzet@univ-lyon1.fr}
\author{Hamza Si Kaddour}
\address{ICJ, Universit\'e de Lyon, Universit\'e Claude Bernard Lyon 1, 43 boulevard du 11 novembre 1918, 69622 Villeurbanne cedex, France}
\curraddr{}
\email{sikaddour@univ-lyon1.fr}
\author{Nicolas Trotignon}
\address{CNRS, LIAFA, Universit\'e Paris Diderot, Paris 7, Case 7014, 75205 Paris Cedex 13, France}
\curraddr{}
\email{nicolas.trotignon@liafa.jussieu.fr}

\thanks{Done under the auspices of the French-Tunisian CMCU "Outils math\'ematiques pour l'Informatique" 05S1505}

\subjclass[2000]{05C60; 05C99.}

\keywords{Graphs; claw-free graphs; cliques;  independent subsets; Paley graphs}

\begin{abstract}
  We describe
${\rm Forb}\{K_{1,3}, \overline {K_{1,3}}\}$, the class of graphs $G$ such that
 $G$ and its complement
$ \overline{G}$ are claw-free. With few exceptions, it is made  of graphs whose connected components consist of cycles of length
at least 4, paths, and of the complements of these graphs. Considering the hypergraph ${\mathcal H} ^{(3)}(G)$ made of the $3$-element subsets of the vertex set of a graph $G$ on
which $G$ induces  a clique or an independent subset, we deduce  from above a  description of the Boolean sum $G\dot{+}G'$  of two graphs  $G$ and $G'$ giving  the same hypergraph. We indicate the role of  this latter description in a reconstruction problem of graphs up to complementation.
\end{abstract}

\maketitle

\section{Results and motivation}
  Our notations and terminology mostly follow \cite {Bo}. The graphs we consider in this paper are  undirected, simple and have no  loop.   That is  a {\it graph} is a pair $G:= (V, \mathcal E)$, where $\mathcal E$ is a subset of $[V]^2$, the set of $2$-element subsets of $V$. Elements of $V$ are the {\it vertices} of $G$ and elements of $\mathcal E$ its {\it edges}.  We denote by $V(G)$ the vertex set of $G$ and  by $E(G)$ its edge set. We look at members of $[V]^2$ as unordered pairs of distinct vertices. If $A$ is a subset of $V$, the pair $G_{\restriction A}:=(A, \mathcal E\cap [A]^2)$ is the \emph{graph induced by $G$ on $A$}. The {\it complement} of $G$ is the simple graph
 ${\overline G}$ whose vertex set is $V$ and whose edges are the unordered pairs of nonadjacent and distinct vertices of $G$, that is $\overline G =(V, {\overline {\mathcal E}})$, where ${\overline {\mathcal E}}=[V]^2\setminus \mathcal E$.
We denote by
$K_3$  the complete graph on
$3$ vertices and  by $K_{1,3}$ the graph  made of a
vertex linked to a
 $\overline {K_{3}}$. The  graph $K_{1,3}$ is called a \emph{claw}, the graph {$\overline {K_{1,3}}$ a {\it co-claw}.

 In \cite{BM}, Brandst{\"a}dt and Mahfud give a structural characterization
of graphs with no claw and no co-claw; they deduce
several algorithmic consequences (relying on bounded clique width). We will give a more precise characterization of such graphs. \\
We denote by $A_{6}$  the graph on $6$ vertices made of a $K_3$  bounded by
three $K_3$ (cf. Figure 1) and by $C_n$ the $n$-element cycle, $n\geq 4$. We denote by
$P_9$ the Paley graph on $9$ vertices (cf. Figure 1). Note that $P_9$ is isomorphic to  its complement
$\overline{P_9}$, to the line-graph of  $K_{3,3}$ and also to $K_3\Box K_3$, the  cartesian  product of $K_3$ by
itself  (see \cite{Bo} page 30 if needed for a definition of the \emph{cartesian product of graphs}, and  see \cite{VW}
page 176 and  \cite {Bo} page 28 for a definition and basic properties of \emph{Paley graphs}).
%FIGURE
\begin{figure}[H]
\begin{center}
\includegraphics[width=4.5in]{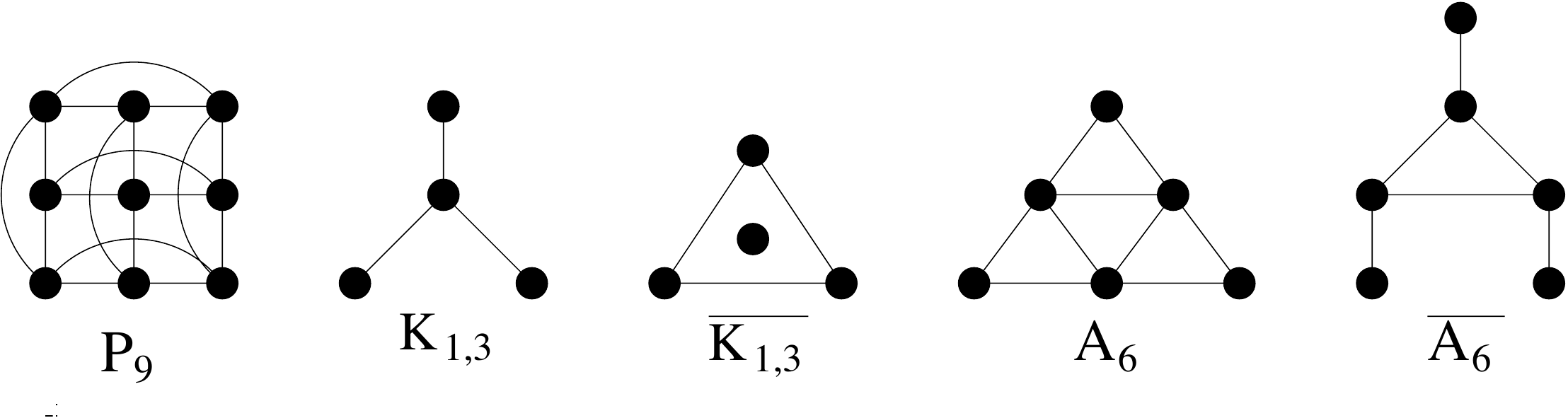}
\end{center}
\caption{} \label{1}
\end{figure}

Given a set  $\mathcal F$ of graphs, we denote by
${\rm Forb} \mathcal F$ the class of graphs $G$ such that
no member of $\mathcal F$ is isomorphic to an
induced subgraph of $G$.    Members of
${\rm Forb} \{K_{3}\}$, resp. ${\rm Forb}\{K_{1,3}\}$ are called  {\it triangle-free}, resp. {\it
claw-free}  graphs.

The main result of this note asserts:

 \begin{theorem} \label{K_{1,3}} The class
${\rm Forb}\{K_{1,3}, \overline {K_{1,3}}\}$ consists of $A_6$; of the induced
subgraphs of $P_9$; of graphs whose connected components are
cycles of length at least $4$ or paths; and of the complements of these graphs.
\end{theorem}
As an immediate consequence of Theorem \ref {K_{1,3}}, note that the
graphs
 $A_{6}$ and  $\overline
{A_{6}}$ are the only members of ${\rm Forb}\{K_{1,3}, \overline
{K_{1,3}}\}$ which contain a $K_{3}$  and  a $\overline
{K_{3}}$ with no vertex in common.
 Note also that $A_{6}$ and  $\overline {A_{6}}$  are very important graphs for the study of how maximal cliques and stable sets overlap in general graphs. See the
main theorem of \cite{XGW2}, see also \cite{XGW1}.
Also, in \cite{Fa}, page 31, a list of all self-complementary
line-graphs is given. Apart from $C_5$, they are all induced subgraphs
of $P_9$.\\

From  Theorem \ref{K_{1,3}} we  obtain a characterization of  the Boolean sum of two graphs having the same $3$-homogeneous subsets. For that, we say that a  subset of vertices of a graph $G$ is {\it
homogeneous} if it is a clique or an independent set (note that the word homogeneous is used with this meaning in Ramsey  theory;  in other areas of graph theory it has other meanings, several in fact).
Let ${\mathcal H}^{(3)}(G)$ be the hypergraph having the same
vertices as
$G$ and whose
hyperedges are the
$3$-element homogeneous subsets of $G$. Given two graphs $G$ and $G'$
on the same vertex set $V$, we recall that the {\it
Boolean sum}
$G\dot{+}G'$ of $G$ and $G'$ is the graph  on $V$ whose edges are unordered pairs $e$ of
distinct vertices such that $e\in E(G)$ if and only if $e\notin E(G')$. Note that $E(G\dot{+}G')$ is the symmetric difference $E(G)\Delta E(G')$ of $E(G)$ and $E(G')$. The graph $G\dot{+}G'$ is also called the \emph{symmetric difference} of $G$ and $G'$ and denoted by $G\Delta G'$ in \cite{Bo}.
Given a graph $U$ with vertex set  $V$, the {\it
edge-graph} of
$U$ is the graph
$S(U)$ whose
vertices are the edges $u$ of $U$ and whose edges are  unordered pairs $uv$
such that
$u=xy$, $v=xz$ for three distinct elements $x,y,z\in V$ such that
$yz$  is
not an edge of $U$. Note that the edge-graph $S(U)$ is a spanning subgraph of $L(U)$, the \emph{line-graph} of $U$, not to be confused with it.

Claw-free graphs and triangle-free graphs are related by means
of the edge-graph
construction. Indeed, as it is immediate to see, for every graph $U$, we have:
$$U\in Forb \{K_{1,3}\}\Longleftrightarrow S(U)\in Forb \{K_{3}\} \ \ \ \ \ \ (\star)$$

Our characterization is this:

\begin{theorem}\label{S(U)}
Let $U$ be a graph. The following properties are equivalent:
\begin{enumerate}[{(1)}]
\item
There are two graphs $G$ and $G'$ having the same
$3$-element homogeneous subsets such that $U:=G\dot{+} G'$;
\item $S(U)$ and $S(\overline U)$ are bipartite;
\item Either (i) $U$ is  an induced subgraph of $P_9$\label{P9},
or (ii)   the connected components of
$U$,  or of its complement $\overline U$,  are cycles of even length or paths. \label{direct}
  \end{enumerate}
\end{theorem}

As a consequence,  if  the  graph  $U$ satisfying Property (1) is disconnected, then $U$ contains no $3$-element cycle\label{claw}, moreover,
if $U$ contains no $3$-element cycle then each connected component of $U$ is  a cycle of even length,  or a path, in particular  $U$ is bipartite.\label{bipartite}

 The implication $(2)\Rightarrow  (3)$  in Theorem \ref{S(U)} follows immediately from Theorem \ref {K_{1,3}}. Indeed,
suppose  that Property (2) holds, that is  $S(U)$ and $S(\overline U)$ are bipartite, then
from Formula ($\star$) and from the fact that $S(A_6)$ and
$S(C_n)$, $n\geq 4$, are respectively isomorphic to $C_9$ and to $C_n$, we have: $$U\in
Forb\{K_{1,3}, \overline {K_{1,3}}, A_6,\overline {A_6}, C_{2n+1}, \overline {C_{2n+1}}
: n\geq 2\}.$$ From Theorem \ref{K_{1,3}}, Property (3) holds. The other implications, obtained by more straigthforward arguments, are given in Subsection \ref{Proof of TheoremS(U)}.

This leaves open the following:
\begin{problem}
Which pairs of graphs $G$ and $G'$  with the same
$3$-element homogeneous subsets have a given Boolean sum $U
:=G\dot{+} G'$?
\end{problem}

A partial answer, motivated by the reconstruction problem discussed below, is given in \cite{dlps2}. We  mention  that two graphs $G$ and $G'$ as above are determined  by the graphs induced on the  connected components of $U
:=G\dot{+} G'$  and on a system of distinct representatives of these connected components (Proposition 10 \cite{dlps2}).

A quite natural problem, related to the study of Ramsey numbers for triples,  is this:
\begin{problem} Which hypergraphs are of the form ${\mathcal H}^{(3)}(G)$ ?
\end{problem}
 An asymptotic  lower bound of the  size of ${\mathcal H}^{(3)}(G)$ in terms of $\vert V(G)\vert $ was established by A.W. Goodman \cite{goo}.

The motivation for Theorem \ref{S(U)} (and thus Theorem \ref {K_{1,3}}) originates in a reconstruction problem on graphs that we present now.
 Considering two graphs $G$ and $G'$ on the same set $V$ of vertices, we say that $G$ and $G'$ are {\it isomorphic up to complementation} if $G'$ is isomorphic to $G$ or to the complement $\overline G$ of $G$. Let $k$ be a non-negative integer,  we say that $G$ and $G'$ are {\it $k$-hypomorphic  up to complementation} if for every $k$-element subset $K$ of $V$, the graphs  $G_{\restriction K}$ and $G'_{\restriction K}$  induced by $G$ and $G'$ on $K$ are isomorphic up to complementation. Finally, we say that $G$ is {\it $k$-reconstructible up to complementation} if every graph  $G'$ which is $k$-hypomorphic to   $G$ up to complementation is in fact isomorphic to $G$ up to complementation.
 The following problem emerged from a question of P.Ille \cite{Ille}:
 \begin{problem} For which pairs $(k,v)$ of integers, $k<v$,   every graph  $G$  on $v$ vertices is $k$-reconstructible up to complementation?
\end{problem}

It is immediate to see that  if the conclusion of the problem above is positive for some $k,v$, then $v$ is distinct from $3$ and $4$ and, with a little bit of thought, that if $v\geq 5$ then $k\geq 4$ (see Proposition 4.1 of \cite{dlps1}).
 With  J. Dammak, G. Lopez \cite{dlps1} and  \cite{dlps2} we proved that the conclusion is positive if: \begin {enumerate} [{(i)}]
\item $4\leq k\leq v-3$  or
\item $4\leq k=v-2$ and $v\equiv 2\;  (mod \; 4)$.  \end{enumerate}
We do not know if in (ii) the condition $v\equiv 2\;  (mod \; 4 )$ can be dropped. For  $4\leq k=v-1$, we checked that the conclusion holds if $v=5$ and noticed that for larger values of $v$ it could be
negative or extremely hard to obtain, indeed, a positive conclusion  would imply that Ulam's reconstruction
conjecture holds (see Proposition 19 of \cite{dlps2}).

The reason for which Theorem \ref{S(U)} plays a role in that matter relies on properties of incidence matrices.

Given non-negative integers $t$, $k$, let $W_{tk}$ be  the $v\choose 2$ by $v\choose k$ incidence matrix of $0$'s and $1$', the rows of which are indexed by $t$-element subsets $T$ of $V$, the colums  are indexed by the $k$-element subsets $K$ of $V$, and where  the entry $W_{tk}(T, K) $ is $1$ if $T\subseteq K$ and is $0$ otherwise.  

Let $U:=G\dot{+}G'$ and  $M_U$ be  the column vector associated to the graph $U$. The matrix product $^{T}W_{2k}M_U$ where the computation is made in the two elements field $\Z/2\Z$ is $0$ if and only if  the number of edges of $G_{\restriction K}$ and $G'_{\restriction K}$ have the same parity for all $K$'s, a condition satisfied if $G$ and $G'$  are $k$-hypomorphic up to complementation and $k\equiv 0\;  (mod \; 4 )$ or $k\equiv 1\;  (mod \; 4 )$. According to  R.M. Wilson 
\cite{W},  the dimension (over $\Z/2\Z$) of the kernel of $^{T}W_{2k}$ is $1$ if $2\leq k\leq v-2$ and $k\equiv 0\;  (mod \; 4 )$ that is $M_U$ is the constant matrix $0$ or $1$, and thus $G'$ is equal to $G$ or to $\overline G$. If  $k\equiv 1\;  (mod \; 4 )$, the dimension is $v$ and the kernel consists of (the colum matrices of) complete bipartite graphs and their complement \cite{dlps1}. If we add the fact that $G$ and $G'$ have the same $3$-homogeneous subsets  then,  according to Theorem \ref{S(U)}, $U$ is claw and co-claw free. If $v\geq 5$,  it follows readily that $U$ is either the empty graph or the complete graph. Hence $G'$ is equal  to $G$ or to $\overline G$. If $3\leq k\leq v-3$, it turns out that  \emph{two graphs $G$ and $G'$ which  are $k$-hypomorphic up to complementation are $3$-hypomorphic up to complementation, which amounts to the fact that $G$ and $G'$ have the same $3$-homogeneous subsets}, thus in the case  $k\equiv 1\;  (mod \; 4 )$,  $G$ and $G'$ are equal up to complementation.   Indeed, a famous  Gottlieb-Kantor theorem on incidence matrices (\cite{Go,KA}) asserts that the matrix $W_{tk}$ has full row rank over the field of rational numbers provided that  $t\leq \min\{k, v-k\}$. From which follows that:

\begin{proposition} \label{down}(Proposition 2.4 \cite{dlps1}) Let  $t \leq min{(k,  v-k)}$ and $G$ and $G'$ be two graphs on the same set $V$ of $v$ vertices.  If $G$ and $G'$ are $k$-hypomorphic up to complementation then they are $t$-hypomorphic up to complementation.
\end{proposition}

Up to now, Wilson theorem has not been applied successfully to the cases $k\equiv 2\;  (mod \; 4 )$ and $k\equiv 0\;  (mod \; 4 )$. Instead, efforts   concentrated on the structure of pairs  of $k$-hypomorphic graphs $G$ and $G'$ with the same $3$-homogeneous subsets.  The  form of their  Boolean sum as given  in (3) of Theorem \ref{S(U)} was the first step of a description. With that in hands, it was shown in \cite{dlps2} that the additional hypothesis that $G$ and $G'$ are $k$-hypomorphic to complementation for some $k$, $4\leq k\leq v-2$,  was enough to ensure that $G$ and $G'$ are isomorphic up to complementation.

\section{Proofs}\label{proofs}
Let $U$ be a graph.
For an unordered pair $e:=xy$ of distinct vertices, we set $U(e)=1$ if
$e\in E(U)$ and $U(e)=0$ otherwise.
Let   $x\in V(U)$; we
denote by
$N_U(x)$ and $d_{U}(x)$ the {\it neighborhood} and
the {\it degree} of
$x$ (that is $N_U(x):=\{ y\in V(U) : xy \in E(U)\}$ and $d_U(x):=\vert N_U(x)\vert$).
For $X\subseteq V(U)$, we
set $N_U(X):= (\cup_{x\in X}N_U(x))\setminus X$.

\subsection {Proof of Theorem \ref{K_{1,3}}.}
Trivially, the graphs described in Theorem \ref{K_{1,3}} belong to ${\rm Forb}\{K_{1,3}, \overline {K_{1,3}}\}$.
We prove the converse.

The {\it diamond} is the graph on four vertices with five edges. We say that {\it a graph} $G$ {\it contains a graph} $H$ when $G$ has an induced subgraph isomorphic to $H$.

\begin{theorem} \label{line-graph} (Harary and Holzmann \cite{HaHo}) A graph $G$ is the line-graph of
a triangle-free graph if and only if $G$ contains no claw and no diamond.
\end{theorem}

 \begin{proof} Since \cite{HaHo} is very difficult to find, we include a short proof. Checking
that a line-graph of a triangle-free graph contains no claw and no diamond is
a routine matter. Conversely, let $G$ be graph with no claw and no diamond.
A theorem of Beineke \cite{Bei1} states that there exists a list $\mathcal L$ of nine graphs such
any graph that does not contain a graph from $\mathcal L$ is a line-graph. One of the
nine graphs is the claw and the eight remaining ones all contain a diamond.
So, $G = L(R)$ for some graph $R$. Let $R'$ be the graph obtained from $R$
by replacing each connected component of $R$ isomorphic to a triangle by a
claw. So, $L(R) = L(R') = G$. We claim that $R'$ is triangle-free. Else let $T$
be a triangle of $R'$. From the construction of $R'$, there is a vertex $v \notin T$ in the connected component of $R'$ that contains $T$. So we may choose $v$ with
a neighbor in $T$. Now the edges of $T$ and one edge from $v$ to $T$ induce a
diamond of $G$, a contradiction. \end{proof}

Let $G$ be in the class ${\rm Forb}\{K_{1,3}, \overline {K_{1,3}}\}$.

\noindent (1) {\it We may assume that $G$ and $\overline G$ are connected}.

\noindent Else, up to symmetry, $G$ is disconnected. If $G$ contains a vertex $v$ of degree
at least $3$, then $N_G(v)$ contains an edge (for otherwise there is a claw), so
$G$ contains a triangle. This is a contradiction since by picking a vertex in
another component we obtain a co-claw. So all vertices of $G$ are of degree at
most $2$. It follows that the components of $G$ are cycles (of length at least $4$,
or there is a co-claw) or paths, an outcome of the theorem. This proves (1).\\

\noindent (2) {\it We may assume that $G$ and $\overline G$ contain  no induced path on six vertices}.\\
Else $G$ has an induced subgraph $H$ that is either a path on at least $6$ vertices
or a cycle on at least $7$ vertices. Suppose $H$ maximal with respect to this
property. If $G = H$ then we are done. Else, by (1), we
pick a vertex $v$ in $G\setminus H$ with at least one neighbor in $H$. From the maximality
of $H$, $v$ has a neighbor $p_i$ in the interior of some $P_6 = p_1p_2p_3p_4p_5p_6$ of
$H$. Up to symmetry we assume that v has a neighbor $p_i$ where $i\in \{2,3\}$.
So $N_G(v) \cap \{p_1, p_2, p_3, p_4\}$ contains an edge $e$ for otherwise $\{p_i, p_{i-1}, p_{i+1}, v\}$ induces a claw. If $e = p_1p_2$ then $v$ must be adjacent to $p_4, p_5, p_6$ for otherwise
there is a co-claw; so $\{ v, p_1, p_4, p_6\}$ induces a claw. If $e = p_2p_3$ then $v$ must be adjacent to $p_5, p_6$ for otherwise there is a co-claw, so from the
symmetry between $\{p_1, p_2\}$ and $\{p_5, p_6\}$ we may rely on the previous case.
If $e = p_3p_4$ then $v$ must be adjacent to $p_1, p_6$ for otherwise there is a co-claw;
so $\{v, p_1, p_3, p_6\}$ induces a claw. In all cases there is a contradiction. This
proves (2).\\

\noindent (3) {\it We may assume that $G$ and $\overline G$ contain  no $A_6$}.\\
Suppose that $G$ contains $\overline{A_6}$. Then, let $aa', bb', cc'$ be three disjoint edges of
$G$ such that the only edges between them are $ab, bc, ca$. If $V (G) = \{a, a', b, b', c, c'\}$, an outcome of the theorem is satisfied, so let $v$ be a seventh
vertex of $G$. We may assume that $av \in E(G)$ (else there is a co-claw). If
$a'v \in E(G)$ then $vb', vc' \in E(G)$ (else there is a co-claw) so $\{v, a', b', c'\}$ is a
claw. Hence $a'v \notin E(G)$. We have $vb \in E(G)$ (or $\{a, a', v, b\}$ is a claw)  and similarly $vc \in E(G)$. So $\{a', v, b, c\}$ is a co-claw. This proves (3).\\

\noindent (4) {\it We may assume that $G$ and $\overline G$ contain  no diamond}.\\
Suppose for a contradiction that $\overline G$ contains a diamond. Then, 
$G$ contains a co-diamond, that is four vertices $a, b, c, d$ that induce
only one edge, say $ab$. By (1), there is a path $P$ from $\{c, d\}$ to some vertex
$w$ that has a neighbor in $\{a, b\}$. We choose such a path $P$ minimal and we
assume up to symmetry that the path is from $c$.

If $w$ is adjacent to both $a, b$ then $\{a, b,w, d\}$ induces a co-claw unless $w$
is adjacent to $d$, similarly $w$
is adjacent to $c$, so $\{w, a, c, d\}$ induces a claw. Hence $w$ is adjacent to exactly one of $a, b$, say to a. So, $P' = cPwab$ is an induced path and for convenience we rename
its vertices $p_1, \cdots , p_k$. If $d$ has a neighbor in $P'$ then, from the minimality
of $P'$, this neighbor must be $p_2$. So, $\{p_2, p_1, p_3, d\}$ induces a claw. Hence, $d$ has
no neighbor in $P'$.

By (1), there is a path $Q$ from $d$ to some vertex $v$ that has a neighbor in
$P'$. We choose $Q$ minimal with respect to this property. From the paragraph
above, $v \neq d$. Let $p_i$ (resp. $p_j$) be the neighor of $v$ in $P'$ with minimum
(resp. maximum) index. If $i = j = 1$ then $dQvp_1Pwp_{k-1}p_{k}$ is a path on  at least
$6$ vertices a contradiction to~(2). So, if $i = j$ then $i \neq 1$ and symmetrically,
$i \neq k$, so $\{p_{i-1}, p_i, p_{i+1}, v\}$ is a claw. Hence $i \neq j$. If $j > i + 1$ then $\{v, v', p_i, p_j\}$, where $v'$ is the neighbor of $v$ along $Q$, is a claw. So, $j = i+1$.
So $vp_ip_j$ is a triangle. Hence $P' = p_1p_2p_3p_4$, $Q = dv$ and $i = 2$, for otherwise
there is a co-claw. Hence, $P' \cup Q$ form an induced $\overline {A_6}$ of $G$, a contradiction
to (3). This proves (4).

Now $G$ is connected and contains no claw and no diamond. So, by
Theorem \ref{line-graph}, $G$ is the line-graph of some connected triangle-free graph $R$.
Symmetrically, $\overline G$ is also a line-graph. These graphs are studied in \cite{Bei2}.

If $R$ contains a vertex $v$ of degree at least $4$ then all edges of $R$ must
be incident with $v$, for else an edge $e$ non-incident with $v$ together with three
edges of $R$ incident with $v$ and non-adjacent to $e$ form a co-claw in $G$. So all
vertices of $R$ have degree at most $3$ since otherwise, $G$ is a clique, a contradiction to (1). We may assume that $R$ has a vertex $a$
of degree $3$ for otherwise $G$ is a path or a cycle. Let $b, b', b''$ be the neighbors
of $a$. Since $a$ has degree $3$, all edges of $R$ must be incident with $b, b'$ or $b''$ for
otherwise $G$ contains a co-claw.

If one of $b, b', b''$, say $b$, is of degree $3$, then $N_R(b) = \{a, a', a''\}$ and all edges
of $R$ are incident with one of $a, a', a''$ (or there is a co-claw). So $R$ is a subgraph
of $K_{3,3}$. So, since $P_9 = L(K_{3,3})$, $G = L(R)$ is an induced subgraph of $P_9$, an
outcome of the theorem. Hence we assume that $b, b', b''$ are of degree at most
$2$. If $\vert N_R(\{b, b', b''\})\setminus \{a\}\vert \geq 3$, then $R$ contains the pairwise non-adjacent edges
$bc, b'c', b''c''$ say, and the edges $ab, ab', ab'', bc, b'c', b''c''$ are vertices of $G$ that
induce an $\overline{A_6}$, a contradiction to (3). So, $\vert N_R(\{b, b', b''\})\setminus \{a\}\vert \leq 2$ which
means again that $R$ is a subgraph of $K_{3,3}$. \hfill $\square$

\subsection{Ingredients for the proof of Theorem \ref{S(U)}.}
The proof of  the equivalence between Properties (1) and (2) of Theorem \ref{S(U)} relies on the following lemma.

 \begin{lemma} \label{G,G',U}   Let $G$ and  $G'$ be two graphs on the same vertex set $V$ and let
  $U:=G\dot{+}G'$. Then, the following properties are equivalent:
  \begin{enumerate}[{(a)}]
 \item     $G$ and $G'$ have the same $3$-element homogeneous subsets;
 \item $U(xy)=U(xz)\neq U(yz) \Longrightarrow   G(xy)\neq G(xz)$ for all distinct elements $x,y,z$ of $V$.
  \item  The sets $A_1:=E(U)\cap E(G)$ and  $A_2:=E(U)\setminus E(G)$  divide
$V(S({U}))$ into two
   independent sets and also
    the sets $B_1:=E({\overline U})\cap E(G)$ and
$B_2:=E({\overline U})\setminus E(G)$  divide
    $V(S({\overline U}))$ into two  independent sets.
 \end{enumerate}  \end{lemma}

 \begin{proof}   Observe first that Property (b) is equivalent to
the conjunction of the following properties:\\
$ (b_{U})$: If $uv$ is an edge of $S(U)$ then $u\in E(G)$ iff      $v\notin E(G)$.            \\
 and\\
$  (b_{\overline U})$: If $uv$ is an edge of $S(\overline U)$ then $u\in E(G)$ iff      $v\notin E(G)$.

\noindent  $(a)  \Longrightarrow  (b)$. Let us show $(a)  \Longrightarrow (b_{U})$.\\
Let  $uv \in E(S(U))$, then $u,v\in E(U)$.
 By contradiction, we may suppose  that $u,v\in E(G)$ (the other case implies $u,v\in E(G')$
  thus is similar).   Since     $u$ and $v$ are edges
 of $U=G\dot{+}G'$ then   $u,v\notin E(G')$. Let $w:=yz$ such that   $u=xy$,  $v=xz$. Then
 $w\notin E(U)$ and thus $w \in E(G)$ iff   $w\in E(G')$. \\
 If $w \in E(G)$,
 $\{x,y,z\}$  is  a homogeneous subset of $G$. Since $G$ and $G'$ have the same $3$-element homogeneous subsets,
  $\{x,y,z\}$  is an homogeneous subset of $G'$. Hence, since $u,v\notin E(G')$,
    $w=yz     \notin E(G')$, thus $w\notin E(G)$, a contradiction.\\
If $w  \notin E(G)$, then $w  \notin E(G')$; since $u,v \notin E(G')$ it follows that $\{x,y,z\}$  is a homogeneous subset of $G'$. Consequently $\{x,y,z\}$  is  a homogeneous subset of $G$.
 Since $u,v\in E(G)$, then $w  \in E(G)$, a contradiction.\\
The implication    $a)  \Longrightarrow (b_{\overline U})$ is similar.\\
$ (b)  \Longrightarrow (a)$. Let $T$ be a $K_{3}$ of $G$.
Suppose that $T$ is not a homogeneous subset of $G'$ then we may suppose
$T=\{u,v,w\}$ with $u,v\in E(G')$ and $w\notin E(G')$
or   $u,v\in E(\overline {G'})$ and
  $w\notin E(\overline {G'})$. In the first case $uv\in E(S(\overline U))$, which contradicts Property
   $ (b_{\overline U})$,
  in the second case     $uv \in E(S(U))$, which contradicts  Property $ (b_{U})$.\\
 $ (b)    \Longrightarrow (c)$. First $V(S(U)) =E(U)=A_1\cup A_2$ and
 $V(S(\overline U))=E(\overline U)=B_1\cup B_2$.
 Let  $u,v$ be two distinct elements
of $A_1$ (respectively  $A_2$). Then  $u,v\in E(G)$ (respectively   $u,v\notin E(G)$). From $ (b_{U})$  we have
 $uv\notin E(S(U))$. Then $A_1$ and $A_2$ are   independent sets of $V(S(U))$.
 The proof  that $B_1$ and $B_2$ are  independent sets of $V(S({\overline U}))$ is similar.\\
$(c)  \Longrightarrow (b)$. This implication is trivial. \end{proof}

\subsection {Proof of Theorem \ref{S(U)}.}\label{Proof of TheoremS(U)}
Implication $ (1)  \Longrightarrow (2)$ follows directly from  implication $(a)  \Longrightarrow (c)$ of Lemma \ref{G,G',U}.  Indeed,  Property (c) implies trivially that  $S(U)$ and $S(\overline U)$ are bipartite. \\
$ (2)  \Longrightarrow (1)$.     Suppose   that $S(U)$ and  $S(\overline U)$  are bipartite.   Let $\{A_1, A_2\}$ and
$\{B_1, B_2\}$  be respectively a partition of   $V(S(U))=E(U)$ and  $V(S(\overline U))=E(\overline U)$ into independent sets. Note that $A_i\cap B_j =\emptyset$, for $i,j\in \{1,2\}$.
Let $G,G'$ be two graphs with the same vertex set as $U$ such that $E(G)=A_1\cup B_1$ and     $E(G')=A_2\cup B_1$.
Clearly  $E(G\dot{+}G')=A_1\cup A_2=E(U)$. Thus $U=G\dot{+}G'$.
To conclude that Property $(1)$ holds, it suffices to  show that $G$ and $G'$ have the same $3$-element homogeneous subsets, that is Property (a) of Lemma \ref{G,G',U} holds. For that, note that $A_1=E(U)\cap E(G)$, $A_2=E(U)\setminus E(G)$, $B_1=E(\overline U)\cap E(G)$ and $B_2=E(\overline U)\setminus E(G)$ and thus Property (c) of
Lemma \ref{G,G',U} holds. It follows that Property (a) of this lemma holds.

The proof of implication $(2)  \Longrightarrow (3)$ was given in Section 1. For the converse implication, let $U$ be a graph satisfying Property (3). It is clear from  Figure 1  that  $S(P{_9})$ is bipartite (vertical edges and horizontal edges form a partition). Since $\overline {P_{9}}$ is isomorphic to $P_9$,  $S(\overline {P{_9}})$ is bipartite too. Thus, if $U$ is isomorphic to an induced subgraph of $P_{9}$, Property (2) holds. If not, we may suppose that  the connected components of $U$ are cycles of even length or paths (otherwise, replace $U$ by $\overline U$). In this case, $S(U)$ is trivially bipartite. In order to prove that Property 2 holds, it suffices to prove that  $S(\overline U)$ is bipartite too. This is a direct consequence of the following claim:

\begin{claim} If $U$ is a bipartite graph, then $S(\overline U)$ is bipartite too.
\end{claim}

\begin{proof}  If $c: V(U) \rightarrow \Z/2\Z$ is a colouring of $U$, set $c':
V(S(\overline U))\rightarrow \Z/2\Z$ defined by $c'(\{x, y\}):= c(x)+c(y)$. \end{proof}

 With this, the proof of Theorem \ref{S(U)} is complete.

\subsection{A direct proof for (3) $\Longrightarrow$ (1) of Theorem \ref{S(U)}.}
In [6] we gave all possible decompositions of a graph $U$ satisfying (1) into a Boolean sum $G\dot{+}G'$ where $G$ and $G'$ have the same $3$-element homogeneous sets.

When $U = P_9$, a decomposition  $U = G\dot{+}G'$
can be given by a picture (see Figure 2).

\begin{figure}[H]
\begin{center}
\includegraphics[width=3in]{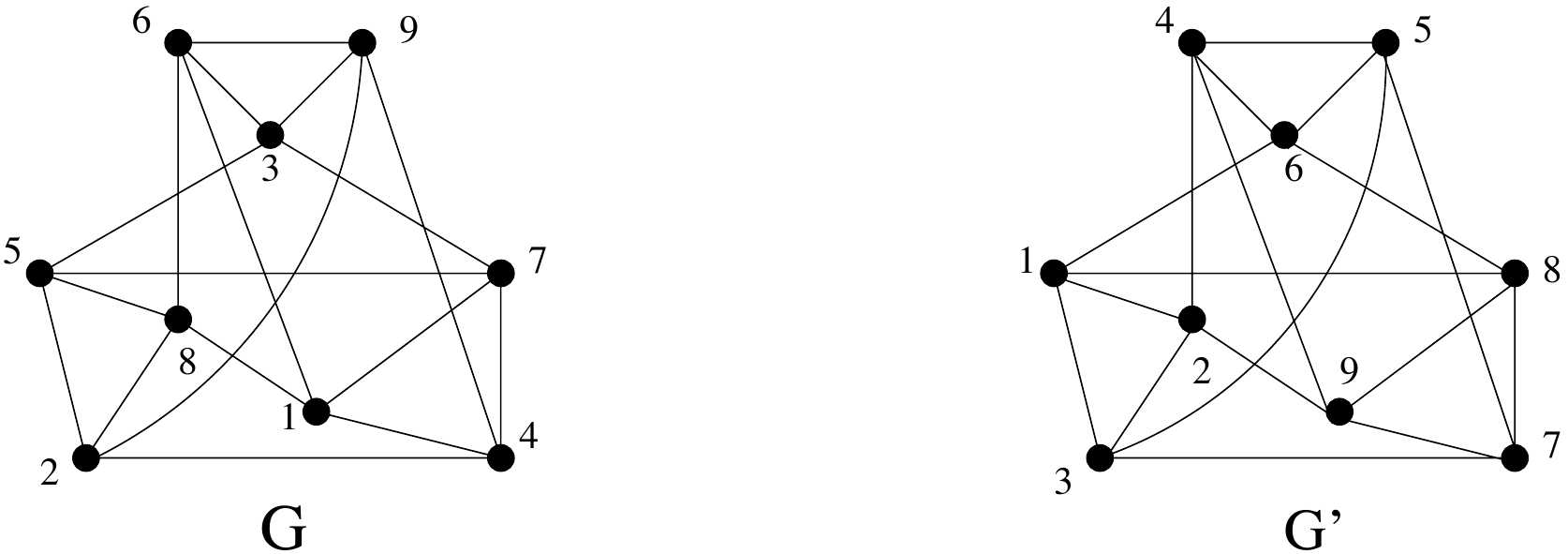}
\end{center}
\caption{} \label{1}
\end{figure}

For the other cases, we introduce the following notation. 

Let $n\geq 2$. Let $X_n$ be an $n$-element set,  $x_0,\cdots , x_{n-1}$ be an enumeration of $X_n$,  $X_n^{0}:= \{x_i\in X_n: i \equiv 0  \ (mod \ 2)\}$ and $X_{n}^{1}:= X_n\setminus X_{n}^{0}$.  Set $R_n:=[X_{n}^{1}]^2\cup [X_n^{2}]^2$, $S_n:=\{\{x_{2i},x_{2i+1}\}:  2i<n\}$, $S'_n:=\{\{x_{2i+1}, x_{2i+2}\}: 2i< n-1\}$. Let $M_n$ and $M'_n$ be the graphs  with vertex set $X_n$ and edge sets $E(M_n):= R_n\cup S_n$ and  $E(M'_n):= R_n\cup S'_n$ respectively. Let $M''_n:= (X_n, R_n\cup S'_n\cup \{\{x_0,x_{n-1}\}\})$ for $n$ even, $n\geq 4$.  For $n\in \{6,7\}$ we give a picture (see Figure 3). For convenience, we set $M_1=M'_1$ the graph with one vertex and we put $V(M_1):=X_1^{0}:=\{x_0\}$. When $G$ is a graph of the form $M_n$, $M'_n$, or $M''_n$, with $n\geq 1$, we put $V^{0}(G):=X_{n}^{0}$ and  $V^{1}(G):=X_{n}^{1}$.  \\

\begin{figure}[H]
\begin{center}
\includegraphics[width=3.4in]{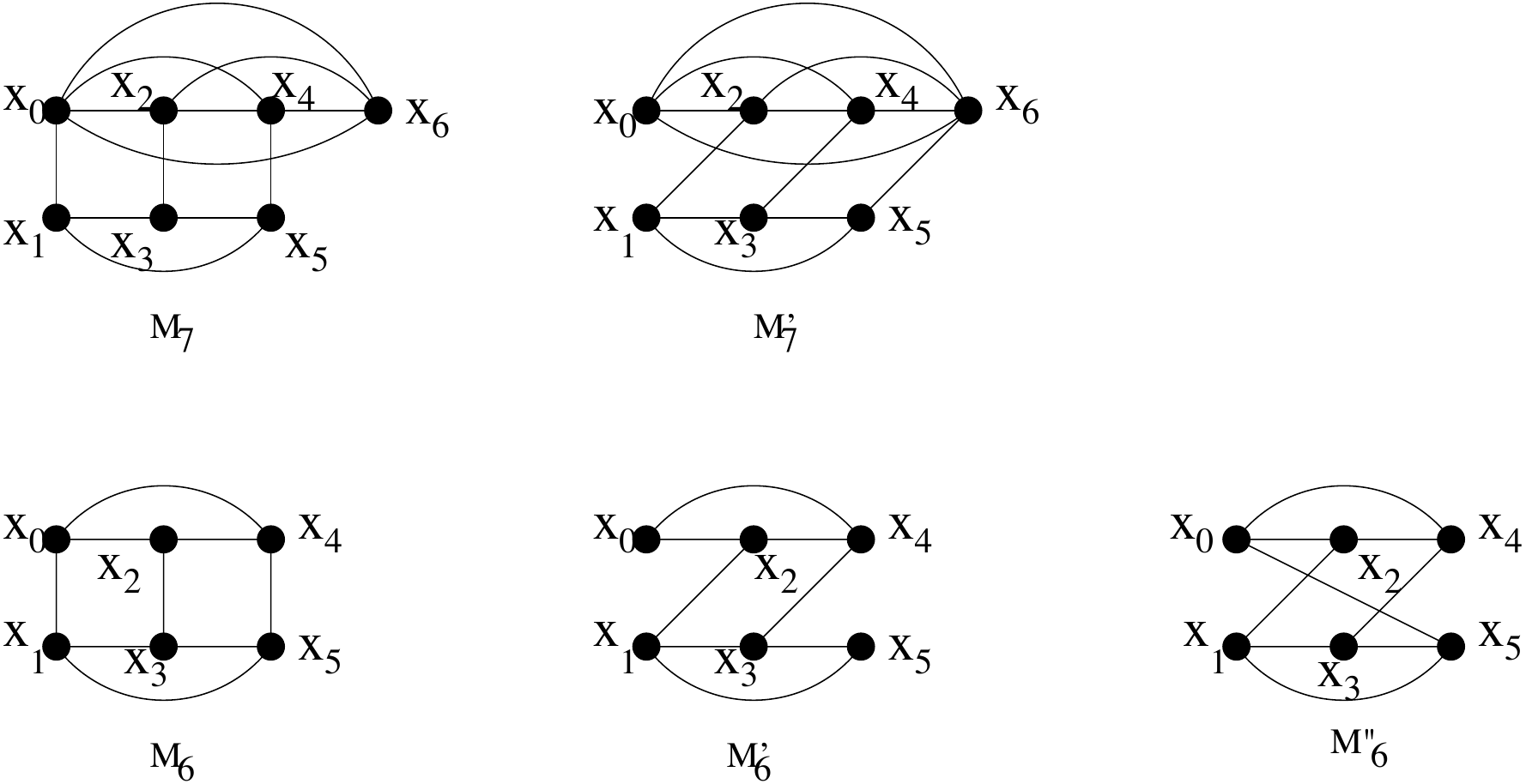}
\caption{} \label{4}
\end{center}
\end{figure}

When $U$ is a cycle of even size $2n$,  a decomposition  $U = G\dot{+}G'$
can be given by  $G=M_{2n}$ and $G'=M''_{2n}$.
When $U$ is a path of size $n$,  a decomposition  $U = G\dot{+}G'$
can be given by  $G=M_n$ and $G'=M'_n$.

When the connected components  of $U$ are cycles of even length or paths,  we define $G$ and $G'$ satisfying $U = G\dot{+}G'$ as follows:   For each connected component  $C$ of $U$, $(G_C, G'_C)$ is given by the previous step. For distinct connected components  $C$ and $C'$ of $U$, $x\in C$, $x'\in C'$, $xx'\in E(G)$ (and $xx'\in E(G')$) if and only if $x\in V^{0}(G_C)$ and $x'\in V^{0}(G_{C'})$, or $x\in V^{1}(G_C)$ and $x'\in V^{1}(G_{C'})$.

When the connected components of $\overline U$ are cycles of even length or paths, from $\overline U = \overline G\dot{+}G'$, the previous step gives a pair $(\overline G ,G')$, then a pair $(G ,G')$.\\

\noindent{\bf Acknowledgements}\\

We  thank S. Thomass\'e  for his helpful comments. We thank the anonymous referee for his careful examination of the paper and his suggestions.


\begin{thebibliography} {99} \label {bibio}
\bibitem {Bei1} L.W. Beineke, Characterizations of derived graphs, J. Combinatorial Theory 9 (1970) 129-135.
\bibitem {Bei2} L.W. Beineke, Derived graphs with derived complements, In Lecture
Notes in Math. (Proc. Conf., New York, 1970) pages 15-24.  Springer (1971).
\bibitem {Bo} J.A. Bondy, U.S.R. Murty, Graph Theory, Basic Graph Theory, Graduate Texts in Mathematics, vol 244,  Springer, 2008, 651 pp.
\bibitem {BM} A. Brandst{\"a}dt and S. Mahfud, Maximum weight stable set on graphs without claw and co-claw (and similar graph classes) can be solved in linear time, Information Processing Letters 84 (2002) 251-259.
\bibitem {dlps1} J. Dammak, G. Lopez, M. Pouzet, H. Si Kaddour, Hypomorphy up to complementation, JCTB, Series B 99 (2009) 84-96.
\bibitem {dlps2} J. Dammak, G. Lopez, M. Pouzet, H. Si Kaddour, Reconstruction of graphs up to complementation, in Proceedings of the First International Conference on Relations, Orders and Graphs: Interaction with Computer Science, ROGICS08, May 12-15 (2008), Mahdia, Tunisia, pp. 195-203.
\bibitem {XGW2} X. Deng, G. Li, W. Zang, Proof of Chv\'atal's conjecture
on maximal stable sets and maximal cliques in graphs,  JCTB, Series B 91 (2004)  301-325.
 \bibitem {XGW1} X. Deng, G. Li, W. Zang, Corrigendum to
proof of Chv\'atal's conjecture on maximal stable sets and maximal cliques
in graphs, JCTB, Series B 94 (2005)   352-353.
 \bibitem {goo} A.W. Goodman, On sets of acquaintances and
strangers at any party, Amer. Math. Monthly 66 (1959) 778-783.
\bibitem {Fa} A. Farrugia. Self-complementary graphs and generalisations: a comprehensive reference manual, Master's thesis, University of Malta (1999).
\bibitem {Go}  D.H. Gottlieb,
\newblock A class of incidence matrices, Proc. Amer. Math. Soc. 17 (1966) 1233-1237.
\bibitem {Ille} P. Ille, personnal communication, September 2000.
\bibitem {HaHo} F. Harary and C. Holzmann, Line graphs of bipartite graphs, Rev. Soc.
Mat. Chile 1 (1974) 19-22.
\bibitem {KA} W. Kantor,  On incidence matrices of finite projective
and affine spaces, Math.Zeitschrift 124 (1972) 315-318.
\bibitem {VW} J.H. Van Lint, R.M. Wilson, A course in
Combinatorics, Cambridge University Press (1992).
\bibitem {W} R.M. Wilson, A Diagonal Form for the Incidence Matrices of $t$-Subsets $vs.$ $k$-Subsets, Europ J. Combinatorics 11 (1990)
  609-615.
\end{thebibliography}
\end{document}